\documentclass{article}

\usepackage[T2A]{fontenc}
\usepackage[utf8]{inputenc}
\usepackage[english]{babel}

\usepackage{amsmath}
\usepackage{amsfonts}
\usepackage{amssymb}
\usepackage{amsthm}
\usepackage{mathrsfs}
\usepackage{hyperref}

\numberwithin{equation}{section}

\renewcommand{\Re}{\mathop{\mathrm{Re}}}

\theoremstyle{plain}
\newtheorem{thm}{Theorem}[section]
\newtheorem{prop}{Proposition}[section]
\newtheorem{lemma}{Lemma}[section]

\theoremstyle{definition}

\newtheorem{definition}{Definition}[section]

\begin{document}

\begin{center}
    \Large\bfseries Characterization and inversion theorems\\ for a generalized Radon transform
\end{center}

\begin{center}
    \it A.\,D. Agaltsov\footnote{119991, Lomonosov Moscow State University, Faculty of Computational Mathematics and Cybernetics, Moscow, Russia; email: agalets@gmail.com}
\end{center}\medskip

\begin{quote}
\centerline{\bf Abstract}

In this paper the generalized Radon transform over level hypersurfaces of CES-functions of measures supported in positive orthant $\mathbb{R}^n_{+}$ is studied. A characterization of the generalized Radon transform of nonnegative measures is found. Explicit inversion formula for the generalized Radon transform of absolutely continuous measures is obtained.\medskip

\end{quote}

\section{Introduction}
One of the actual problems in economic statistics is a need for taking into account the effect of substitution of production factors at the microlevel in production processes. Some production processes require resourses which are able to interchange one another. Appearance of this effect was accelerated by the processes of standartisation and globalisation. In the work~[Sh] A.~Shananin introduced the generalized model of pure industry which takes into account the described effect. The profit function in this model is closely related to a generalized Radon transform of measures. Study of properties of profit functions is important since the profit function is one of the main tools of macroscopic description of production systems.

Define on the set of nonegative numbers $a,b$ the operation
\begin{equation*}
    a \oplus_{\alpha} b = \left( a^{\alpha} + b^{\alpha} \right)^{\frac{1}{\alpha}}, \; 0 < \alpha \leqslant 1.
\end{equation*}
For $\alpha = 1$ it is the usual addition. Now define an analog of scalar product for two nonnegative vectors $p,x \in \mathbb R^n_+$ by the rule
\begin{equation*}
    p \odot_{\alpha} x = p_1 x_1 \oplus_{\alpha} p_2 x_2 \oplus_{\alpha} \ldots \oplus_{\alpha} p_n x_n.
\end{equation*}
In economics the map $x \mapsto p \odot_\alpha x$ is called constant elasticity of substitution (CES) function. In this paper we study the generalized Radon transform of signed Borel measures $\mu$ with support in positive orthant $\mathbb R^n_+$ which is given by the formula
\begin{equation}
    \label{def.radon}
    R_\alpha[\mu](p,p_0) = \frac{\partial}{\partial p_0} \int\limits_{ p \odot_{\alpha} x \leqslant p_0} \mu(dx), 
\end{equation}
where $p \in \mathbb R^n_+$, $p \neq 0$, $p_0 > 0$ and the derivative is taken in the sense of distribution theory. We also study the transform 
\begin{equation}
    \label{def.prod}
    {\Pi}_{\alpha}[\mu](p,p_0) = \int\limits_{\mathbb{R}^{n}_{+}} (p_0 - p \odot_{\alpha} x)_{+} \mu(dx),
\end{equation}
where $a_+ = \max(0,a)$. This transform plays the role of profit function in the generalized model of pure industry~[Sh]. Transforms~\eqref{def.radon} and~\eqref{def.prod} are closely related. Results obtained for one of them could be easily translated to the case of the other transform.

For transforms~\eqref{def.radon} and~\eqref{def.prod} we obtain inversion and characterization theorems. Before passing to these questions we show that in the case of absolutely continuous measures with continuous densities the generalized Radon transform $R_\alpha[\mu](p,p_0)$ is precisely an integral of the density over a level hypersurface of a CES-function.

\begin{prop}
Let a signed measure $\mu$ on $\mathbb R^n_+$ be absolutely continuous with continuous density $a(x)$. Let $\Omega_\alpha(p)$ be a differential form on $\mathbb R^n_+$ such that
\begin{equation*}
    d_{x}(p \odot_{\alpha} x) \wedge \Omega_{\alpha}(p) = dx_1 \wedge \ldots \wedge dx_n, \;\;\; p \neq 0.
\end{equation*}
Then the generalized Radon transform satisfies 
\begin{equation*}
    R_\alpha[\mu](p,p_0) = \int\limits_{ p \odot_{\alpha} x = p_0 } a(x) \, \Omega_{\alpha}(p), \;\;\; p \neq 0, \; p_0 > 0.
\end{equation*} 
\end{prop}

\begin{proof}
By the coarea formula (see [KP]) we have
\begin{equation*}
    \int\limits_{ p \odot_{\alpha} x \leqslant p_0} a(x) \, dx = \int\limits_{0}^{p_0} \int\limits_{ p \odot_{\alpha} x = s} a(x) \, \Omega_{\alpha}(p) \, ds.
\end{equation*}
Taking into account the definition of the generalized Radon transform we obtain
\begin{equation*}
    R_\alpha[\mu](p,p_0) = \frac{\partial}{\partial p_0} \int\limits_{ p \odot_{\alpha} x \leqslant p_0 } a(x) \, dx = \int\limits_{ p \odot_{\alpha} x = s} a(x) \, \Omega_{\alpha}(p).
\end{equation*}
Proposition is proved.
\end{proof}

\section{Characterization.}
Before passing to the characterization theorems we give some definitions that will be used in future.

\begin{definition} A distribution $T \in \mathscr D'(0,\infty)$ is called nonnegative if for any test function $\varphi \in \mathscr D(0,\infty)$, $\varphi \geqslant 0$, we have $\langle T, \varphi \rangle \geqslant 0$.
\end{definition}

\begin{definition} Let $X_1 = (\Omega_1,\mathcal{F}_1)$ and $X_2 = (\Omega_2,\mathcal{F}_2)$ be two measurable spaces, $f \colon \Omega_1 \to \Omega_2$ be a measurable map, $\mu$ be a measure on $X_1$. Then the measure $f_\ast \mu$ on $X_2$ defined for each $F_2 \in \mathcal{F}_2$ by the rule $f_\ast \mu(F_2) = \mu(f^{-1}(F_2))$ is called pushforward of measure $\mu$.
\end{definition}

\begin{definition} A function $F \in C^{\infty}( \mathop{\mathrm{int}} \mathbb{R}^{n}_{+}, \mathbb{R})$ is called completely monotone if for any $\alpha \in \mathbb{Z}^n_{+}$ the inequality
\begin{equation*}
    (-1)^{|\alpha|} \frac{\partial^{|\alpha|} F(p)}{\partial^{\alpha} p} \geqslant 0,
\end{equation*}
holds for any $p \in \mathop{\mathrm{int}} \mathbb{R}^{n}_{+}$.
\end{definition}

The proofs of our characterization theorems are based on the next two theorems.

\begin{thm}[see \lbrack LL\rbrack] \label{agal.th.nonneg}
 Let $T \in \mathscr D'(0,\infty)$ be a nonnegative distribution. Then there is a unique locally finite Borel measure $\nu$ on $(0,\infty)$ such that for any $\varphi \in \mathscr D(0,\infty)$ the following equality holds:
\begin{equation*}
    \langle T, \varphi \rangle = \int\limits_{0}^{\infty} \varphi(\tau) \, \nu(d\tau).
\end{equation*}
\end{thm}

\begin{thm}[S. Bernstein, V. Hilbert, see \lbrack HS\rbrack] \label{agal.th.bern}
Let $F(p) \colon \mathbb{R}^n_{+} \to \mathbb{R}$ be a bounded on $\mathbb{R}^n_{+}$ function such that $F(p)$ is completely monotone on $\mathop{\mathrm{int}} \mathbb{R}^n_{+}$. Then there is a unique measure $\mu$ supported in $\mathbb{R}^n_{+}$ such that for any $p \in \mathbb{R}^n_{+}$ we have
\begin{equation*}
    F(p) = \int\limits_{\mathbb{R}^n_{+}} e^{-p_1 x_1 - \ldots - p_n x_n} \mu(dx).
\end{equation*}
\end{thm}

Now we pass to the characterization theorems. The next theorem characterizes the generalized Radon transforms of finite nonnegative measures supported in $\mathbb R^n_+$.

\begin{thm}\label{agal.th.charrad} A distribution $\ae(p,\cdot) \in \mathscr D'(0,\infty)$, $p \in \mathbb{R}^{n}_{+}$, $p \neq 0$ can be represented in the form
    \begin{equation*}
        \ae(p,p_0) = R_\alpha[\mu](p,p_0),
    \end{equation*}
where $\mu$ is a nonnegative finite Borel measure supported in $\mathbb{R}^{n}_{+}$ and absolutely continuous at zero (i.e. $\mu(\{0\})=0$) if and only if
\begin{enumerate}
    \item $\ae(p,\cdot) \geqslant 0$,
    \item $\lambda \ae(\lambda p, \lambda p_0) = \ae(p,p_0)$ for any $\lambda>0$,
    \item the function
       \begin{equation*}
          F(p) = \int\limits_{0}^{\infty} e^{-\tau^{\alpha}} \ae(p_1^{\frac{1}{\alpha}},\ldots,p_n^{\frac{1}{\alpha}};\tau) d\tau,
       \end{equation*}
       is bounded on $\mathbb{R}^n_{+}$ and completely monotone on $\mathop{\mathrm{int}} \mathbb{R}^{n}_{+}$ and $\lim\limits_{\lambda \to +\infty}F(\lambda p)=~0$ for any $p \neq 0$.
\end{enumerate}
\end{thm}

\begin{proof}\textit{Necessity.}
Let's obtain, at first, a formula for the action of distrubution $\ae(p,p_0)$ on functions $\varphi \in \mathscr D( \mathop{\mathrm{int}} \mathbb{R}^n_{+})$. We are going to show that
\begin{equation}
    \label{ae.action}
    \langle \ae(p,\tau), \varphi(\tau) \rangle = \int\limits_{\mathbb{R}^n_{+}} \varphi(p \odot_{\alpha} x) \mu(dx).
\end{equation}
By definition of derivative in the sense of distribution theory we can write
\begin{multline*}
    \langle \ae(p,\tau), \varphi(\tau) \rangle = - \left\langle \int\limits_{p \odot_{\alpha} x \leqslant \tau} \mu(dx), \frac{\partial \varphi(\tau)}{\partial \tau} \right\rangle = -\left\langle \int\limits_{\mathbb{R}^n_{+}} \theta(\tau- p \odot_{\alpha} x) \, \mu(dx), \frac{\partial \varphi(\tau)}{\partial \tau} \right \rangle = \\
    = -\int\limits_{0}^{\infty} \int\limits_{\mathbb{R}^{n}_{+}} \theta(\tau - p \odot_{\alpha} x) \, \mu(dx) \frac{\partial \varphi(\tau)}{\partial \tau} \, d\tau = -\int\limits_{\mathbb{R}^{n}_{+}} \int\limits_{0}^{\infty} \theta(\tau - p \odot_{\alpha} x) \frac{\partial \varphi(\tau)}{\partial \tau} \, d\tau \, \mu(dx) = \\
    = - \int\limits_{\mathbb{R}^n_{+}} \int\limits_{p \odot_{\alpha} x}^{\infty} \frac{\partial \varphi(\tau)}{\partial \tau} \, d\tau \, \mu(dx) = \int\limits_{\mathbb{R}^n_{+}} \varphi(p \odot_{\alpha} x) \, \mu(dx).
\end{multline*}
where $\theta(\cdot)$ is the Heaviside function. The property $\ae(p,\cdot) \geqslant 0$ follows immediately from formula~\eqref{ae.action}. Now let's note that distribution $\ae(p,p_0)$ can be extended to the set of functions $\psi \in C[0,\infty)$, $\psi|_{(0,\infty)} \in C^\infty(0,\infty)$, with compact support in $[0,\infty)$. Considering the case when $\ae(p,p_0)$ is a continuous function of $p_0$ we can correct the formula~\eqref{ae.action}:
\begin{equation}
    \label{ae.genaction}
    \langle \ae(p,\tau), \psi(\tau) \rangle = \mu\{0\} \psi(0) + \int\limits_{\mathbb{R}^n_{+}} \psi(p \odot_{\alpha} x) \, \mu(dx).
\end{equation}
By virtue of absolute continuity of $\mu$ at zero the additional term vanishes. Hence formula~\eqref{ae.action} is valid for functions of class $C^{\infty}(0,\infty) \cap C[0,\infty)$ with compact support in $[0,\infty)$. 

For any $\lambda > 0$ the obviuous equality holds:
\begin{equation*}
    \int\limits_{ (\lambda p) \odot_{\alpha} x \leqslant \lambda p_0} \mu(dx) = \int\limits_{ p \odot_{\alpha} x \leqslant p_0} \mu(dx).
\end{equation*}
Hence $\ae(p,p_0)$ is a positively homogenuous distribution of order $-1$ as a derivative of function that is positively homogenuous of order $0$. More precisely, for any function $\varphi \in \mathscr D(0,\infty)$ we can write
\begin{gather*}
    \langle \ae(\lambda p, \lambda \tau), \varphi(\tau) \rangle = \frac{1}{\lambda} \left\langle \ae(\lambda p, \tau), \varphi \left( \frac{\tau}{\lambda} \right) \right\rangle = \\
    = -\frac{1}{\lambda^2} \left\langle \int\limits_{(\lambda p) \odot_{\alpha} x \leqslant \tau} \mu(dx), \varphi'\left( \frac{\tau}{\lambda}\right) \right\rangle = -\frac{1}{\lambda^2} \left\langle \int\limits_{ p \odot_{\alpha} x \leqslant \frac{\tau}{\lambda}} \mu(dx), \varphi'\left( \frac{\tau}{\lambda}\right) \right\rangle = \\
    = - \frac{1}{\lambda} \left\langle \int\limits_{ p \odot_{\alpha} x \leqslant \tau} \mu(dx), \varphi'(\tau) \right\rangle =
    \left\langle \frac{1}{\lambda} \ae(p,\tau), \varphi(\tau) \right\rangle.
\end{gather*}
The property $\lambda \ae(\lambda p, \lambda p_0) = \ae(p,p_0)$ is shown.

Finally, taking into account formula~\eqref{ae.action} and finiteness of measure~$\mu$ we obtain
\begin{equation*}
    F(p) = \left\langle \ae(p_1^{\frac{1}{\alpha}},\ldots,p_n^{\frac{1}{\alpha}};\tau), e^{-\tau^{\alpha}} \right\rangle = \\
    \int\limits_{\mathbb{R}^{n}_{+}} e^{-p_1 x_1^{\alpha} - \ldots - p_n x_n^{\alpha}} \, \mu(dx),
\end{equation*}
From this we can derive boundedness and complete monotony of function $F(p)$ and the property $\lim\limits_{\lambda \to +\infty} F(\lambda p)= \mu(\{0\}) = 0$. The necessary part of theorem is proved.

\textit{Sufficiency.} \textit{Step 1.} 
From the condition $\ae(p,\cdot) \geqslant 0$ by virtue of Theorem~\ref{agal.th.nonneg} for any $p \in \mathbb{R}^n_{+}$, $p \neq 0$, there exists and unique a nonnegative locally-finite Borel measure $\hat{\mu}_{p}$ with support in $(0,\infty)$ such that for any function $\varphi \in \mathscr D(0,\infty)$ the following equality holds:
\begin{equation*}
    \langle \ae(p,\tau), \varphi(\tau) \rangle = \int\limits_{0}^{\infty} \varphi(\tau) \, \hat{\mu}_{p}(d\tau).
\end{equation*}
Let's use positive homogeneity of distribution $\ae(p,p_0)$ to compute for $\lambda > 0$ the value 
\begin{gather*}
    \langle \ae(\lambda p, \tau), \varphi(\tau) \rangle = \left\langle  \frac{\lambda}{\lambda}  \ae\left( \lambda p, \frac{\lambda}{\lambda} \tau \right), \varphi\left( \frac{\lambda}{\lambda} \tau \right) \right\rangle = \\ 
\left\langle  \frac{1}{\lambda}  \ae\left( p, \frac{1}{\lambda} \tau \right), \varphi\left( \frac{\lambda}{\lambda} \tau \right) \right\rangle = \langle \ae(p,\tau), \varphi(\lambda \tau) \rangle.
\end{gather*}

Now let $\chi_{n}(\tau) \in \mathscr D(0,\infty)$ be a nondecreasing sequence of functions such that $\chi_{n}(\tau) = 1$ for $\tau \in [\frac{1}{n},n]$ and support of $\chi_{n}$ is contained in $[\frac{1}{2n},2n]$. Define $\phi_{n}(\tau) = e^{-{\tau}^{\alpha}} \chi_{n}(\tau)$. Then by the Lebesgue theorem on monotone convergence the map $\tau \mapsto e^{-{\tau}^{\alpha}}$ is integrable with respect to measure $\hat{\mu}_{p}$ and
\begin{equation*}
    \langle \ae(p,\tau), \phi_{n}(\tau) \rangle = \int\limits_{0}^{\infty} \varphi_{n}(\tau) \, \hat{\mu}_{p}(d\tau) \to \int\limits_{0}^{\infty}e^{-\tau^{\alpha}} \, \hat{\mu}_{p}(d\tau) = F(\lambda p_1^{\alpha}, \ldots,\lambda p_n^{\alpha}).
\end{equation*}

Let's make a substitution $s = {\tau}^{\alpha}$. Then
\begin{equation*}
    F(\lambda p_1^{\alpha}, \ldots, \lambda p_n^{\alpha}) = \int\limits_{0}^{\infty} e^{-\lambda s} \, \hat{\mu}_{p*}(ds),
\end{equation*}
where $\hat{\mu}_{p*}$ is the pushforward of measure $\hat{\mu}_{p}$ under the map $\tau \mapsto {\tau}^{\alpha}$.\\[.5em]
\noindent\textit{Step 2.} From boundedness of $F(p)$ on $\mathbb{R}^{n}_{+}$ and complete monotony on $\mathop{\mathrm{int}}\mathbb{R}^{n}_{+}$ by Theorem~\ref{agal.th.bern} there exists nonnegative finite Borel measure $\mu_*$ supported in $\mathbb{R}^{n}_{+}$ such that
\begin{equation*}
    F(p) = \int\limits_{\mathbb{R}^{n}_{+}} e^{- p_1 x_1 - \ldots -p_n x_n} \, \mu_*(dx).
\end{equation*}
Let $\mu$ be the pushforward of measure $\mu_*$ under the map $(x_1,\ldots,x_n) \mapsto (x_1^{1/\alpha},\ldots,x_n^{1/\alpha})$. Then $\mu_*$ is the pushforward of $\mu$ under the map $(x_1,\ldots,x_n) \mapsto (x_1^\alpha,\ldots,x_n^\alpha)$. Measure $\mu$ is finite since measure $\mu_*$ is finite. After the change of variables in the integral we obtain
\begin{equation}
    \label{fpvianu}
    F(p) = \int\limits_{\mathbb{R}^{n}_{+}} e^{ -p_1 x_1^{\alpha} - \ldots - p_n x_n ^{\alpha}} \, \mu(dx).
\end{equation}
We can write
\begin{equation*}
    F(p) = \mu(\{ 0 \}) + \int\limits_{\mathbb{R}^n_{+} \setminus \{0\}} e^{-p_1 x_1^{\alpha} - \ldots - p_n x_n^{\alpha}} \, \mu(dx).
\end{equation*}
Put $p = \lambda q$, $q \neq 0$ and pass $\lambda \to +\infty$. Using the monotone convergence theorem and condition $\lim\limits_{\lambda \to +\infty}F(\lambda q)=0$ we obtain $\mu(\{0\}) = 0$, i.e. measure $\mu$ is absolutely continuous at zero.

Function
\begin{equation*}
    p_0 \mapsto \int\limits_{p \odot_{\alpha} x \leqslant p_0} \mu(dx)
\end{equation*}
is nondecreasing and bounded. Denote by $\widetilde{\mu}_p$ the Lebesgue--Stieltjes measure generated by this function on $[0,\infty)$. Define a distribution $\ae_0(p,p_0) = R_\alpha[\mu](p,p_0)$. Then for any continuous and bounded function $\varphi$ on $[0,\infty)$ the equality holds:
\begin{equation*}
    \langle \ae_{0}(p,\tau), \varphi(\tau) \rangle =  \int\limits_{0}^{\infty} \varphi(\tau) \, \widetilde{\mu}_{p}(d\tau).
\end{equation*}
Put $\varphi(\tau) = e^{-\tau^{\alpha}}$. From formulas \eqref{ae.action}, \eqref{fpvianu} it follows that
\begin{equation*}
    \langle \ae_{0}(p,\tau), e^{-\tau^{\alpha}} \rangle = \int\limits_{\mathbb{R}^n_{+}} e^{-(p_1 x_1)^{\alpha} - \ldots - (p_n x_n)^{\alpha}} \, \mu(dx) = F(p_1^{\alpha},\ldots,p_n^{\alpha}).
\end{equation*}

From the necessary part of this theorem it follows that $\ae_0(p,p_0)$ is a positively homogeneous distribution of order~$-1$. By the same considerations as in step one we obtain
\begin{gather*}
    F(\lambda p_1^{\alpha},\ldots,\lambda p_n^{\alpha}) = \langle \ae_{0}(\lambda p, \tau), e^{-\tau^{\alpha}} \rangle = \\  
    = \langle \ae_{0}(p,\tau), e^{-\lambda \tau^{\alpha}} \rangle = \int\limits_{0}^{\infty} e^{-\lambda \tau^{\alpha}} \, \widetilde{\mu}_{p}(d\tau). 
\end{gather*}
Making the change of variables $s=\tau^\alpha$ and denoting by $\widetilde{\mu}_{p*}$ the pushforward of measure~$\widetilde{\mu}_p$ under the map $\tau \mapsto \tau^\alpha$ we obtain
\begin{equation*}
    F(\lambda p_1^{\alpha},\ldots,\lambda p_n^{\alpha}) = \int\limits_{0}^{\infty} e^{-\lambda s} \, \widetilde{\mu}_{p*}(ds). 
\end{equation*}
From the finiteness of measure $\widetilde{\mu}_p$ we derive the finiteness of measure $\widetilde{\mu}_{p*}$.\\[.5em]
\noindent\textit{Step 3.}
At steps 1 and 2 we obtained that for any $\lambda>0$
\begin{equation*}
    \int\limits_{0}^{\infty} e^{-\lambda s} \, \hat{\mu}_{p*}(ds) = F(\lambda p_1^{\alpha},\ldots,\lambda p_n^{\alpha}) = \int\limits_{0}^{\infty} e^{-\lambda s} \, \widetilde{\mu}_{p*}(ds),
\end{equation*}
Pass $\lambda \to +0$. Using the monotone convergence theorem for the integral with respect to measure $\hat{\mu}_{p*}$ we obtain that $\widetilde{\mu}_{p*}$ is finite and the equality
\begin{equation*}
    \int\limits_{0}^{\infty} e^{-\lambda s} \, \hat{\mu}_{p*}(ds) =  \int\limits_{0}^{\infty} e^{-\lambda s} \, \widetilde{\mu}_{p*}(ds)
\end{equation*}
holds for $\lambda \geqslant 0$. From this we derive the coincidence of measures~$\hat{\mu}_{p*}$ и $\widetilde{\mu}_{p*}$. 
Hence the measures $\hat{\mu}_{p}$ и $\widetilde{\mu}_{p}$ generating distributions $\ae(p,p_0)$ and $\ae_{0}(p,p_0)$ also coincide. Finally,
\begin{equation*}
    \ae(p,p_0) = \ae_{0}(p,p_0) = R_\alpha[\mu](p,p_0).
\end{equation*}
Theorem is proved.
\end{proof}

We will apply Theorem~\ref{agal.th.charrad} in order to prove the characterization theorem for transform~$\Pi_\alpha[\mu](p,p_0)$. The next theorem characterises transform~$\Pi_\alpha[\mu](p,p_0)$ in the case of nonnegative absolutely continuous at zero measures supported in $\mathbb R^n_+$. Recall that in the generalized model of pure industry~[Sh] the transform~$\Pi_\alpha[\mu](p,p_0)$ has the meaning of the production function. The measure $\mu$ in this case has the meaning of distribution of production powers over technologies. Within this context the requirement of absolute continuity at zero means absense of <<the horn of plenty>> or impossibility to have profit without spending any resources. Hence the requirement~$\mu(\{0\})=0$ is not restrictive from the point of view of economic applications. Before formulating the next theorem we prove the lemma that will help us in future.

\begin{definition} Let $\mu$ be a signed measure supported in $\mathbb R^n_+$ and let $\mu = \mu_{+} - \mu_{-}$ be its Jordan decomposition. Then the measure $|\mu|=\mu_{+} + \mu_{-}$ is called total variation of measure~$\mu$. 
\end{definition}

\begin{lemma}\label{agal.lm.proddif}
Let $\mu$ be a signed Borel measure supported in $\mathbb R^n_+$ for which the total variation $|\mu|$ is finite on compacts. Then for any $p_0 > 0$, $p \in \mathbb R^n_+ \setminus \{0\}$ the function $\Pi_\alpha[\mu](p,p_0)$ is differentiable with respect to $p_0$ and the equality holds:
\begin{equation*}
    \frac{\partial \Pi_{\alpha}[\mu](p,p_0)}{\partial p_0} = \int\limits_{p \odot_{\alpha} x \leqslant p_0} \mu(dx).
\end{equation*}
\end{lemma}

\begin{proof}
Denote $G(p,p_0) = \left\{ x \in \mathbb{R}^n_{+} \mid p \odot_{\alpha} x \leqslant p_0 \right\}$. Let $\Delta > 0$. Let's write
\begin{gather*}
    \Pi_{\alpha}[\mu](p,p_0+\Delta) - \Pi_{\alpha}[\mu](p,p_0) = \\
    = \int\limits_{G(p,p_0+\Delta)}(p_0 + \Delta - p \odot_{\alpha} x) \, \mu(dx) - \int\limits_{G(p,p_0)}(p_0 - p \odot_{\alpha} x) \, \mu(dx) = \\
    = \Delta \int\limits_{G(p,p_0+\Delta)}\mu(dx) + \int\limits_{G(p,p_0+\Delta) \setminus G(p,p_0)} (p_0 - p \odot_{\alpha} x) \, \mu(dx).
\end{gather*}
 Note that
 \begin{equation*}
    G(p,p_0+\Delta)\setminus G(p,p_0) = \left\{ x \in \mathbb{R}^n_{+} \mid 0 < p \odot_{\alpha} x - p_0 \leqslant \Delta \right\}.
 \end{equation*}
Hence the estimate holds:
\begin{gather*}
    \left| \int\limits_{G(p,p_0+\Delta) \setminus G(p,p_0)} (p_0 - p \odot_{\alpha} x) \, \mu(dx) \right| \leqslant \\
    \leqslant \int\limits_{G(p,p_0+\Delta) \setminus G(p,p_0)} |p_0 - p \odot_{\alpha} x| \, |\mu|(dx) \leqslant \Delta \int\limits_{G(p,p_0+\Delta)\setminus G(p,p_0)} |\mu|(dx) = o(\Delta), \; \Delta \to +0,
\end{gather*}
since $|\mu| \left\{ G(p,p_0+\Delta) \setminus G(p,p_0) \right\} \to 0$ as $\Delta \to +0$. Hence one can write
\begin{equation*}
    \frac{1}{\Delta}\left[ \Pi_{\alpha}[\mu](p,p_0+\Delta)-\Pi_{\alpha}[\mu](p,p_0) \right] = \int\limits_{G(p,p_0+\Delta)}\mu(dx) + o(1), \; \Delta \to +0.
\end{equation*}
Using the absolute continuity of the Lebesgue integral we obtain
\begin{equation*}
    \int\limits_{G(p,p_0+\Delta)}\mu(dx) \to \int\limits_{G(p,p_0)}\mu(dx), \; \Delta \to +0.
\end{equation*}
The case $\Delta<0$ is similar to the considered one. Lemma is proved.
\end{proof}

Note that in the work [HS] the characterization theorem for transform~$\Pi_\alpha[\mu](p,p_0)$ was obtained in the case~$\alpha=1$. For $\alpha=1$ the generalized Radon transform $R_\alpha[\mu](p,p_0)$ coincides with the classical Radon transform over hyperplanes. From the point of view of economical applications the profit function $\Pi_\alpha[\mu](p,p_0)$ in the case~$\alpha=1$ corresponds to ecomonic systems without the effect of substitution of production factors at the microlevel. The next theorem generalizes the described result to the case of arbitrary~$\alpha \in (0,1]$.

\begin{thm}\label{agal.th.chprod} A function~$\Pi(p,p_0) \colon \mathbb{R}^{n}_{+} \times (0,\infty) \to [0,\infty)$ can be represented in the form
\begin{equation*}
    \Pi(p,p_0) = \Pi_{\alpha}[\mu](p,p_0),
\end{equation*}
where $\mu$ is a nonnegative Borel measure with support in $\mathbb R^n_+$ and such that $\mu(\{0\}) = 0$ if and only if
\begin{enumerate}
    \item $\Pi(p,p_0)$ is convex,
    \item $\Pi(\lambda p, \lambda p_0) = \lambda \Pi(p,p_0)$ for any $\lambda > 0$, 
    \item $\Pi(p,+0) = \frac{\partial \Pi}{\partial p_0}(p,+0) = 0$ for $p \in \mathop{\mathrm{int}}\mathbb{R}^n_{+}$, 
    \item The function
    \begin{equation*}
        F(p) = \int\limits_{0}^{\infty} e^{-\tau^{\alpha}} \, d \frac{\partial \Pi}{\partial \tau} (p_1^{\frac{1}{\alpha}},\ldots,p_n^{\frac{1}{\alpha}};\tau )
    \end{equation*}
    is bounded in $\mathbb{R}^{n}_{+}$ and completely monotone in $\mathop{\mathrm{int}} \mathbb{R}^{n}_{+}$.
\end{enumerate}
\end{thm}

\begin{proof}\textit{Necessity.}
Convexity of function~$\Pi_\alpha[\mu](p,p_0)$ follows immediately from its definition if we take into account that for $\alpha \in (0,1]$ the function $x \mapsto p \odot_\alpha x$ is concave (for $\alpha > 1$ it becomes convex).

The positive homogeneity follows immediately from definition of transform $\Pi_\alpha[\mu](p,p_0)$:
\begin{equation*}
    \Pi_{\alpha}[\mu](\lambda p, \lambda p_0) = \int\limits_{\mathbb{R}^n_{+}} (\lambda p_0 - (\lambda p) \odot_{\alpha} x)_{+} \, \mu(dx) = \lambda \int\limits_{\mathbb{R}^n_{+}} (p_0 - p \odot_{\alpha} x)_{+} \, \mu(dx) = \lambda \Pi_{\alpha}[\mu](p,p_0).
\end{equation*}

By virtue of inequality $(p_0 - p \odot_{\alpha} x)_{+} \leqslant p_0$ and of finiteness of measure $\mu$ we obtain that
\begin{equation*}
    0 \leqslant \Pi_{\alpha}[\mu](p,p_0) \leqslant p_0 \int\limits_{\mathbb{R}^{n}_{+}} \mu(dx) \to 0, \; p_0 \to +0.
\end{equation*}
It follows that $\Pi_{\alpha}[\mu](p,+0) = 0$. 

Using Lemma~\ref{agal.lm.proddif} we can write
\begin{equation*}
    \frac{\partial \Pi_{\alpha}[\mu](p,p_0)}{\partial p_0} = \mu(\left\{ 0 \right\}) + \int\limits_{ p \odot_{\alpha} x \leqslant p_0, \; x \neq 0} \mu(dx).
\end{equation*}
Passing $p_0 \to +0$, taking into account the absolute continuity of the Lebesgue integral and the property $\mu(\{0\})=0$ we obtain $\frac{\partial \Pi_{\alpha}[\mu]}{\partial p_0}(p,+0)=0$.

Boundedness and complete monotony of function $F(p)$ can be proved as in Theorem~$\ref{agal.th.charrad}$ if one takes into account that
\begin{equation*}
    \frac{\partial^2 \Pi_{\alpha}[\mu]}{\partial p_0^2}(p,p_0) = R_\alpha[\mu](p,p_0).
\end{equation*}
The necessary part of the theorem is proved.

\textit{Sufficiency. }
Let $\ae(p,p_0) = \frac{\partial^2 \Pi(p,p_0)}{\partial p_0^2}$ be a derivative in the sense of distribution theory i.e. for any $\varphi \in \mathscr D(0,\infty)$ the equality holds:
\begin{equation*}
    \langle \ae(p,\tau), \varphi(\tau) \rangle = \langle \Pi(p,\tau), \varphi''(\tau) \rangle.
\end{equation*}
A function on $(0,+\infty)$ is convex if and only if its second derivative is a nonnegative distribution (see [Sc]). Hence $\ae(p,\cdot) \geqslant 0$.

Since $\Pi(p,p_0)$ is a positively homogeneous function of order one its second derivative $\ae(p,p_0)$ is a positively homogenuous distribution of order~$-1$:
\begin{gather*}
    \langle \ae(\lambda p, \lambda \tau), \varphi(\tau) \rangle = 
    \frac{1}{\lambda} \left\langle \ae( \lambda p, \tau ), \varphi \left(\frac{\tau}{\lambda}\right) \right\rangle = 
    \frac{1}{\lambda^3} \left\langle \Pi(\lambda p, \tau), \varphi'' \left(\frac{\tau}{\lambda}\right) \right\rangle = \\
    = \frac{1}{\lambda^2} \langle \Pi(\lambda p, \lambda \tau), \varphi''(\tau) \rangle = \frac{1}{\lambda} \langle \Pi(p,\tau), \varphi''(\tau) \rangle = 
    \left\langle \frac{1}{\lambda} \ae(p,\tau), \varphi(\tau) \right\rangle.
\end{gather*}
Finally,
\begin{equation*}
    F(p) = \int\limits_{0}^{\infty} e^{-\tau^{\alpha}} \, d \frac{\partial \Pi}{\partial \tau} (p_1^{\frac{1}{\alpha}},\ldots,p_n^{\frac{1}{\alpha}};\tau) = \int\limits_{0}^{\infty} e^{-{\tau}^{\alpha}} \ae(p_1^{\frac{1}{\alpha}},\ldots,p_n^{\frac{1}{\alpha}};\tau) \, d\tau.
\end{equation*}
As in the proof of Theorem~\ref{agal.th.charrad} it can be shown that for any $\lambda>0$ the equality holds:
\begin{equation*}
    F(\lambda p) =  \int\limits_{0}^{\infty} e^{-\lambda \tau^{\alpha}} \, d \frac{\partial \Pi}{\partial \tau} (p_1^{\frac{1}{\alpha}},\ldots,p_n^{\frac{1}{\alpha}};\tau)
\end{equation*}
Pass $\lambda \to +\infty$, use the Lebesgue theorem on monotone convergence and the condition $\frac{\partial \Pi}{\partial p_0}(p,+0) = 0$ to obtain that $F(\lambda p) \to 0$ as $\lambda \to +\infty$.

Now let's use Theorem~\ref{agal.th.charrad}. By this theorem there exists a nonnegative Borel absolutely continuous at zero measure $\mu$ supported in $\mathbb R^n_+$ such that $\ae(p,p_0) = R_\alpha[\mu](p,p_0)$. Define $\Pi_{0}(p,p_0) = \Pi_{\alpha}[\mu](p,p_0)$. Then
\begin{equation*}
    \frac{\partial^2 \Pi(p,p_0)}{\partial p_0^2} = R_\alpha[\mu](p,p_0) = \frac{\partial^2 \Pi_{0}(p,p_0)}{\partial p_0^2}.
\end{equation*}
From coincidence of second derivatives it follows that functions $\Pi(p,p_0)$ and $\Pi_{0}(p,p_0)$ differ by a polynomial of degree at most one with respect to $p_0$ (see [Sc]) for any fixed~$p$. But from the equalities
\begin{gather*}
    \Pi(p,+0) = \Pi_{0}(p,+0) = 0, \\
    \frac{\partial}{\partial p_0} \Pi(p,+0) = \frac{\partial}{\partial p_0} \Pi_{0}(p,+0) = 0,
\end{gather*}
it follows that this polynomial is equal to zero. The theorem is proved. 
\end{proof}

\section{Inversion.}
Let $\mu$ be an absolutely continuous measure on $\mathbb R^n_+$ with density $a(x)$. Denote
\begin{gather*}
    R_\alpha[a](p,p_0) := R_\alpha[\mu](p,p_0), \\ 
    \Pi_{\alpha}[a](p,p_0) := \Pi_{\alpha}[\mu](p,p_0).
\end{gather*}
We also use the notation $L^p(\mathbb{R}^{n}_{+},\rho(x))$ for the class of functions that belong to $L^p (\mathbb R^n_+)$ with weight $\rho(x)$, i.e. of such measurable functions $f(x)$ that
\begin{equation*}
    \int\limits_{\mathbb{R}^{n}_{+}} |f(x)|^p \rho(x) \, dx < \infty.
\end{equation*}
Before formulating the inversion theorem we prove an auxilary lemma.
\begin{lemma}\label{agal.lm.integral} 1. Let $\Omega$ be a differential form on $\mathbb R^n_+$ such that
\begin{equation*}
(dx_1 + \ldots + dx_n) \wedge \Omega_{\alpha} = dx_1 \wedge \ldots \wedge dx_n,
\end{equation*}
Then for $\Re z_1 > 0$, \dots, $\Re z_n > 0$ the following equality holds:
\begin{equation*}
    \int\limits_{ x_1 + \ldots + x_n = 1, x \geqslant 0 } x_{1}^{z_1-1} \ldots x_{n}^{z_n-1} \, \Omega = B(z) = \frac{\Gamma(z_1)\ldots \Gamma(z_n)}{\Gamma(z_1+\ldots+z_n)}, \; z =(z_1,\ldots,z_n).
\end{equation*}

\noindent 2. For $\Re t_1 < 1$, \dots, $\Re t_n < 1$ the following formula holds:
\begin{equation*}
    \int\limits_{\mathbb{R}^n_{+}} u_1^{-t_1} \ldots u_n^{-t_n} \left( 1 - \left( u_1 + \ldots + u_n \right)^{\frac{1}{\alpha}} \right)_{+} \, du = \alpha B(1-t)B \left(2,\alpha(n-t_1-\ldots-t_n) \right),
\end{equation*}
where $t = (t_1,\ldots,t_n)$, $du = du_1 \ldots du_n$.
\end{lemma}

\begin{proof}1) Note that
\begin{equation*}
    \int\limits_{\mathbb{R}^n_{+}} x_1^{z_1-1}\ldots x_n^{z_n-1} e^{-x_1-\ldots-x_n}\, dx = \int\limits_{0}^{\infty} x_1^{z_1-1} e^{-x_1} \, dx_1 \ldots \int\limits_{0}^{\infty} x_n^{z_n-1} e^{-x_n} \, dx_n = \Gamma(z_1)\ldots \Gamma(z_n).
\end{equation*}
On the other hand
\begin{gather*}
    \int\limits_{\mathbb{R}^n_{+}} x_1^{z_1-1} \ldots x_n^{z_n-1} e^{-x_1 - \ldots -x_n} \, dx = \left\{ \text{coarea formula} \right\} = \\
    = \int\limits_{0}^{\infty} e^{-s} \int\limits_{x_1+\ldots+x_n = s, \; x\geqslant 0} x_1^{z_1 - 1} \ldots x_n^{z_n - 1} \, \Omega \, ds = \left\{ x_k = y_k s, \; k = \overline{1,n} \right\} = \\
    = \int\limits_{0}^{\infty} e^{-s} s^{z_1 + \ldots + z_n - 1} \, ds \, \int\limits_{y_1 + \ldots + y_n = 1, \; y \geqslant 0} y_1^{z_1-1}\ldots y_n^{z_n-1} \, \Omega = \\
    = \Gamma(z_1 + \ldots + z_n)  \int\limits_{y_1 + \ldots + y_n = 1, \; y \geqslant 0} y_1^{z_1-1}\ldots y_n^{z_n-1} \, \Omega.
\end{gather*}
Comparing the obtained expressions we obtain
\begin{equation*}
    \int\limits_{y_1 + \ldots + y_n = 1, \; y \geqslant 0} y_1^{z_1-1}\ldots y_n^{z_n-1} \, \Omega = \frac{\Gamma(z_1)\ldots \Gamma(z_n)}{\Gamma(z_1+\ldots+z_n)} = B(z).
\end{equation*}

\noindent 2) Consider the following chain of transformations:
\begin{gather*}
    \int\limits_{\mathbb{R}^n_{+}} u_1^{-t_1} \ldots u_n^{-t_n} \left( 1 - \left( u_1 + \ldots + u_n \right)^{\frac{1}{\alpha}} \right)_{+} \, du = \left\{ \text{coarea formula} \right\} = \\
    = \int\limits_{0}^{1} \left( 1 - s^{\frac{1}{\alpha}} \right) \int\limits_{u_1 + \ldots + u_n = s, \; u\geqslant 0} u_1^{-t_1} \ldots u_n^{-t_n} \, \Omega \, ds = 
    \left\{ u_k = v_k s, \; k = \overline{1,n} \right\} = \\
    = \int\limits_{0}^{1} \left( 1 - s^{\frac{1}{\alpha}} \right) s^{n-1-t_1-\ldots-t_n} \, ds \int\limits_{v_1+\ldots+v_n=1, \; v \geqslant 0} v_1^{-t_1}\ldots v_n^{-t_n} \, \Omega.
\end{gather*}
From the first statement of the lemma it follows that
\begin{equation*}
    \int\limits_{v_1+\ldots+v_n=1, \; v \geqslant 0} v_1^{-t_1}\ldots v_n^{-t_n} \, \Omega = B(1-t_1,\ldots,1-t_n) =: B(1-t).
\end{equation*}
Let $\gamma$ be a real number. Let's evalute the integral
\begin{gather*}
    \int\limits_{0}^{1} \left( 1 - s^{\frac{1}{\alpha}} \right) s^{\gamma-1} \, ds = 
    \left\{ t = s^{\frac{1}{\alpha}}, \; s = t^{\alpha}, \; ds = \alpha t^{\alpha-1}dt \right\} = \\
    = \alpha \int\limits_{0}^{1} (1-t) t^{ \alpha(\gamma-1) + \alpha - 1} dt = 
    \alpha \int\limits_{0}^{1} (1-t)^{2-1} t^{ \alpha \gamma - 1} dt = \alpha B(2,\alpha \gamma).
\end{gather*}
Hence
\begin{equation*}
    \int\limits_{0}^{1} \left( 1 - s^{\frac{1}{\alpha}} \right) s^{n-1-t_1-\ldots-t_n} \, ds = \alpha B \left(2,\alpha(n-t_1-\ldots-t_n)\right).
\end{equation*}
The second part of the lemma is proved.
\end{proof}

Note that in the work [HS] an inversion formula for transform $\Pi_{\alpha}[\mu](p,p_0)$ was obtained in the case of $\alpha = 1$. Now we are going to prove a theorem that generalizes this result to the case of arbitrary $\alpha \in (0,1]$. From the point of view of economic applications the next theorem can be used to find distributions of powers over technologies given the profit function in a production system with effect of substution of production factors at the microlevel, when it is a priori known that the distribution is absolutely continuous.

\begin{thm}\label{agal.th.inv} Let
\begin{equation*}
    a(x) \in L^1 \left(\mathbb{R}^{n}_{+},x_1^{\alpha(c_1-1)}\ldots x_n^{\alpha(c_n-1)} \right) \cap L^2 \left(\mathbb{R}^{n}_{+},x_1^{2\alpha(c_1-1)+1}\ldots x_n^{2\alpha(c_n-1)+1}\right)
\end{equation*}
for some real $c_1 < 1$, \dots, $c_n < 1$. Then
\begin{gather*}
    a(x) = \int\limits_{\mathbb{R}^{n}_{+}} K_{\alpha}(p_1 x_1,\ldots,p_n x_n;c) \Pi_{\alpha}[a](p,1) \, dp ,\\
    \left( a(x) = \int\limits_{\mathbb{R}^n_{+}} \int\limits_{0}^{1} \int\limits_{0}^{t}K_{\alpha}(p_1 x_1,\ldots,p_n x_n;c) R_\alpha[a](p,s)\,ds\,dt\,dp,  \right)
\end{gather*}
where the kernel $K_\alpha$ is
\begin{gather*}
    K_{\alpha}(u;c) = \frac{\alpha^{2n-1}}{(2\pi i)^n} \lim\limits_{R \to \infty} \int\limits_{c+i\mathbf{B}^n(0,R)} \frac{u_1^{-\alpha(z_1-1)-1}\ldots u_n^{-\alpha(z_n -1)-1} dz}{B(1-z_1,\ldots,1-z_n)B\left(2,\alpha(n - z_1 - \ldots - z_n) \right)}, \\
    c = (c_1,\ldots,c_n),
\end{gather*}
and $\mathbf{B}^n(0,R)$ is a ball in $\mathbb R^n$ of radius~$R$ with center at the origin.
\end{thm}

\begin{proof}
For brevity denote $\Pi(p,p_0) = \Pi_\alpha[a](p,p_0)$. By definition
\begin{equation*}
    \Pi(p,p_0) = \int\limits_{\mathbb{R}^n_{+}} \left(p_0 - \left((p_1 x_1)^{\alpha}+\ldots + (p_n x_n)^{\alpha} \right)^{\frac{1}{\alpha}} \right)_{+} a(x) \, dx.
\end{equation*}
Make change of variables $y_k = x_k^{\alpha}$, $k=\overline{1,n}$. The Jacobian of this transformation $x_k \mapsto y_k$ is equal to $\alpha^{-n}(x_1\ldots x_n)^{1-\alpha}$. Denoting $a_{*}(x_1^{\alpha},\ldots,x_n^{\alpha}) = {\alpha}^{-n}a(x_1,\ldots,x_n)(x_1 \ldots x_n)^{1-\alpha}$ we can write
\begin{equation*}
    \Pi(p,p_0) = \int\limits_{\mathbb{R}^n_{+}} \left( p_0 - \left( p_1^{\alpha} y_1 + \ldots + p_n^{\alpha} y_n \right)^{\frac{1}{\alpha}} \right)_{+} a_{*}(y) \, dy.
\end{equation*} 
Using this formula and the Fubini theorem we obtain
\begin{gather*}
    \int\limits_{\mathbb{R}^n_{+}} p_1^{-t_1} \ldots p_n^{-t_n} \Pi(p_1^{\frac{1}{\alpha}},\ldots,p_n^{\frac{1}{\alpha}}; 1) \, dp = \\
    = \int\limits_{\mathbb{R}^n_{+}} a_{*}(y)  \int\limits_{\mathbb{R}^n_{+}} p_1^{-t_1} \ldots p_n^{-t_n} \left( 1 - \left( p_1 y_1 + \ldots + p_n y_n \right)^{\frac{1}{\alpha}} \right)_{+} \, dp \, dy = \left\{ u_k = p_k y_k, \, k=\overline{1,n} \right\} = \\
    = \int\limits_{\mathbb{R}^n_{+}} y_{1}^{t_1-1}\ldots y_{n}^{t_n-1} a_{*}(y) \, dy \int\limits_{\mathbb{R}^n_{+}} u_1^{-t_1} \ldots u_n^{-t_n} \left(1-  \left( u_1 + \ldots + u_n \right)^{\frac{1}{\alpha}} \right)_{+} \, du.
\end{gather*}
From the latter formula it is obvious that we can obtain the Mellin transform of function~$a_\ast(x)$ using the function~$\Pi(p,p_0)$. The inversion formula is based on this fact.

From Lemma~\ref{agal.lm.integral} it follows that for $\Re t_1 < 1$, \dots, $\Re t_n < 1$ the following formula holds:
\begin{equation}
    \label{inv.factorformula}
    \begin{array}{c}
    \int\limits_{\mathbb{R}^n_{+}} p_1^{-t_1} \ldots p_n^{-t_n} \Pi(p_1^{\frac{1}{\alpha}},\ldots,p_n^{\frac{1}{\alpha}}; 1) \, dp = \\
    = \alpha B(1-t)B \left(2,\alpha\left( n - \sum\limits_{k=1}^{n} t_k \right) \right) \int\limits_{\mathbb{R}^n_{+}} y_{1}^{t_1-1}\ldots y_{n}^{t_n-1} a_{*}(y) \, dy.
    \end{array}
\end{equation}
Put $\tau = (\tau_1,\ldots,\tau_n) \in \mathbb{R}^n$. Let's prove the next formula for the Fourier tranform:
\begin{equation}
    \label{inv.fourier}
\mathcal{F} \left(a_{*}\left(e^{x_1},\ldots,e^{x_n}\right) \exp(c \cdot x) \right)(\tau) = 
\int\limits_{\mathbb{R}^n_{+}} y_{1}^{i\tau_1 + c_1-1} \ldots y_{n}^{i\tau_n + c_n-1} a_{*}(y_1,\ldots,y_n) \, dy.
\end{equation}
We have the following chain of transformations:
\begin{gather*}
    \mathcal{F} \left(a_{*}\left(e^{x_1},\ldots,e^{x_n}\right) \exp(c_1 x_1 + \ldots + c_n x_n) \right)(\tau) = \\
    = \int\limits_{\mathbb{R}^n} \exp\left( i(x_1 \tau_1 + \ldots +x_n \tau_n) + c_1 x_1 + \ldots + c_n x_n \right) a_{*}\left( e^{x_1}, \ldots, e^{x_n} \right) \, dx = \\
    = \left\{ y_k = e^{x_k}, \; x_k = \ln{y_k}, \; dx_k = y_{k}^{-1} dy_k \; k = \overline{1,n} \right\} = \\
    = \int\limits_{\mathbb{R}^n_{+}} y_{1}^{i\tau_1 + c_1-1} \ldots y_{n}^{i\tau_n + c_n-1} a_{*}(y_1,\ldots,y_n) \, dy.
\end{gather*}
Take into account that
\begin{gather*}
    \int\limits_{\mathbb{R}^n} \exp\left( c_1 x_1 + \ldots + c_n x_n \right) \left| a_{*}\left( e^{x_1}, \ldots, e^{x_n} \right) \right| \, dx = \\
    = \int\limits_{\mathbb{R}^n_{+}} y_{1}^{c_1-1} \ldots y_{n}^{c_n-1} \left|a_{*}(y_1,\ldots,y_n) \right| \, dy = \left\{ y_k = u_k^{\alpha}, \; dy_k = \alpha u_{k}^{\alpha-1} du_k, \; k = \overline{1,n} \right\} = \\
    = \int\limits_{\mathbb{R}^n_{+}} u_{1}^{\alpha (c_1-1)}\ldots u_{n}^{\alpha (c_n-1)} |a(u_1,\ldots,u_n)| \, du < \infty,
\end{gather*}
by virtue of the conditions of this theorem. Similarly,
\begin{gather*}
    \int\limits_{\mathbb{R}^n} \exp\left( 2( c_1 x_1 + \ldots + c_n x_n) \right) \left| a_{*}\left( e^{x_1}, \ldots, e^{x_n} \right) \right|^2 \, dx = \\
    =\int\limits_{\mathbb{R}^n_{+}} y_{1}^{2c_1-1}\ldots y_{n}^{2c_n-1} |a_{*}(y_1,\ldots,y_n)|^2 \, dy = \\
    =\int\limits_{\mathbb{R}^n_{+}} u_{1}^{\alpha(2c_1-1)}\ldots u_{n}^{\alpha(2c_n-1)}|a(u_1,\ldots,u_n)|^2 \alpha^{-2n} (u_1 \ldots u_n)^{2-2\alpha} \alpha^{n} (u_1\ldots u_n)^{\alpha-1} \, du = \\
    = \alpha^{-n} \int\limits_{\mathbb{R}^n_{+}} u_{1}^{2\alpha(c_1-1)+1}\ldots u_{n}^{2 \alpha(c_n-1)+1} |a(u_1,\ldots,u_n)|^2 \, du < \infty.
\end{gather*}
Thereby
\begin{equation*}
     \exp\left( c_1 x_1 + \ldots + c_n x_n \right) a_{*}\left( e^{x_1}, \ldots, e^{x_n} \right) \in L^1(\mathbb{R}^n) \cap L^2(\mathbb{R}^n).
\end{equation*}
From this it follows that the Fourier transform~\eqref{inv.fourier} exists and, by the Plancherel theorem (see [Yo]), belongs to~$L^2(\mathbb{R}^n)$.

Taking into account~\eqref{inv.factorformula} and~\eqref{inv.fourier} we obtain
\begin{gather*}
    \mathcal{F} \left(a_{*}\left(e^{x_1},\ldots,e^{x_n}\right) \exp(c_1 x_1 + \ldots + c_n x_n) \right)(\tau) = \\
     = \int\limits_{\mathbb{R}^n_{+}} \frac{ p_1^{-i\tau_1-c_1} \ldots p_n^{-i\tau_n - c_n} \, \Pi (p_1^{\frac{1}{\alpha}},\ldots,p_n^{\frac{1}{\alpha}}; 1) \, dp}{\alpha B(1-c-i\tau)B \left(2,\alpha n - \alpha(c_1 + i\tau_1 + \ldots + c_n + i\tau_n) \right)}.
\end{gather*}
Apply the inverse Fourier transform to the latter equation. Since the function above belongs to~$L^2(\mathbb{R}^n)$ the inverse Fourier transform can be computed using the following formula (see [Yo]):
\begin{gather*}
    a_{*}\left(e^{x_1},\ldots,e^{x_n}\right) \exp(c_1 x_1 + \ldots + c_n x_n) = \\
    = \frac{1}{(2\pi)^n} \lim\limits_{R \to \infty} \int\limits_{\mathbf{B}(0,R)} \int\limits_{\mathbb{R}^n_{+}} \frac{ \exp( - i x \cdot \tau) p_1^{-i\tau_1-c_1} \ldots p_n^{-i\tau_n - c_n} \, \Pi (p_1^{\frac{1}{\alpha}},\ldots,p_n^{\frac{1}{\alpha}}; 1) \, dp}{\alpha B(1-c-i\tau)B \left(2,\alpha n - \alpha(c_1 + i\tau_1 + \ldots + c_n + i\tau_n) \right)}  d\tau. 
\end{gather*}
Denote $y_k = e^{x_k}$, $k=\overline{1,n}$. The latter formula becomes
\begin{gather*}
    a_{*}(y_1,\ldots,y_n) = \\
    = \frac{1}{(2\pi)^n} \lim\limits_{R \to \infty} \int\limits_{\mathbf{B}(0,R)} \int\limits_{\mathbb{R}^n_{+}} \frac{ (p_1 y_1)^{-i\tau_1-c_1} \ldots (p_n y_n)^{-i\tau_n - c_n} \, \Pi (p_1^{\frac{1}{\alpha}},\ldots,p_n^{\frac{1}{\alpha}}; 1) \, dp}{\alpha B(1-c-i\tau)B \left(2,\alpha n - \alpha(c_1 + i\tau_1 + \ldots + c_n + i\tau_n) \right)}  d\tau. 
\end{gather*}
Now denote $y_{k} = x_{k}^{\alpha}$ and make the substitution $p_{k} = q_{k}^{\alpha}$, $k = \overline{1,n}$. Turning back from function $a_{*}(\cdot)$ to function $a(\cdot)$, we obtain
\begin{equation*}
    a(x) = \frac{\alpha^{2n-1}}{(2\pi)^n} \lim\limits_{R \to \infty} \int\limits_{\mathbf{B}(0,R)} \int\limits_{\mathbb{R}^n_{+}} \frac{ (q_1 x_1)^{-\alpha(i\tau_1+c_1-1)-1} \ldots (q_n x_n)^{-\alpha(i\tau_n + c_n-1)-1} \, \Pi (q_1,\ldots,q_n; 1) \, dq}{\alpha B(1-c-i\tau)B \left(2,\alpha n - \alpha(c_1 + i\tau_1 + \ldots + c_n + i\tau_n) \right)}  d\tau. 
\end{equation*}
Making substitution~$z_k = c_k + i \tau_{k}$, $k = \overline{1,n}$ and using the definition of kernel $K_\alpha(u;c)$ in the statement of this theorem, we obtain the required formula for inversion of transform~$\Pi_\alpha[a](p,p_0)$. The inversion formula for transform~$R_\alpha[a](p,p_0)$ immediately follows from equalities
\begin{gather}
    \label{inv.pitorad}
    \frac{\partial^2 \Pi_{\alpha}[a](p,p_0)}{\partial p_0^2} = R_\alpha[a](p,p_0), \\
    \label{inv.piinitcond}
    \Pi_{\alpha}[a](p,+0) = \frac{\partial \Pi_{\alpha}[a]}{\partial p_0}(p,+0) = 0, \;\;\; p \in \mathop{\mathrm{int}} \mathbb{R}^n
\end{gather} 
that imply $\Pi_{\alpha}[a](p,p_0) = \int_{0}^{p_0}\int_{0}^{t} R_\alpha[a](p,s) \, ds \, dt$. Let's show that equalities~\eqref{inv.pitorad}--\eqref{inv.piinitcond} hold.

The equality~\eqref{inv.pitorad} follows from the Lemma~\ref{agal.lm.proddif} and from the definition of transform~$R_\alpha[a](p,p_0)$.

Next, for any $p \in \mathop{\mathrm{int}} \mathbb{R}^n_{+}$ there exist such $R(p) > 0$ and $p_{0}^{*}(p)$ that
\begin{equation*}
    \Pi_{\alpha}[a](p,p_0) = \int\limits_{\mathbb{R}^n_{+} \cap \mathbf{B}^n(0,R(p))} (p_0 - p \odot_{\alpha} x)_{+} a(x) \, dx,
\end{equation*}
for $0<p_0<p_{0}^{*}(p)$. Write
\begin{gather*}
    0 \leqslant \left| \Pi_{\alpha}[a](p,p_0) \right| \leqslant \int\limits_{\mathbb{R}^n_{+} \cap \mathbf{B}^n(0,R(p))} (p_0 - p \odot_{\alpha} x)_{+} |a(x)| \, dx \leqslant \\ 
    \leqslant  p_{0} \int\limits_{\mathbb{R}^n_{+} \cap \mathbf{B}^n(0,R(p))} |a(x)| \, dx 
    \leqslant C p_{0} \int\limits_{\mathbb{R}^n_{+} \cap \mathbf{B}^n(0,R(p))} x_{1}^{\alpha(c_1-1)} \ldots x_{n}^{\alpha(c_n-1)} |a(x)|\,dx \leqslant \\
    \leqslant C p_{0} \| a(x) \|_{L^1 \left(\mathbb{R}^n_{+},x_{1}^{\alpha(c_1-1)} \ldots x_{n}^{\alpha(c_n-1)} \right)}
\end{gather*}
where $C>0$ is some constant. Passing $p_0 \to +0$ we obtain that $\Pi_{\alpha}[a](p,+0)~=~0$.

Now use Lemma~\ref{agal.lm.proddif} and obtain
\begin{gather*}
    0 \leqslant \left|\frac{\partial \Pi_{\alpha}[a](p,p_0)}{\partial p_0} \right| \leqslant \int\limits_{ p \odot_{\alpha} x \leqslant p_0} |a(x)| \, dx \leqslant \\
    \leqslant \int\limits_{ \mathbb{R}^n_{+} \cap \mathbf{B}^n(0,C_{1} p_0)} |a(x)| \, dx \leqslant 
    C_{2} \int\limits_{ \mathbb{R}^n_{+} \cap \mathbf{B}^n(0,C_{1} p_0)} x_1^{\alpha(c_1-1)} \ldots x_n^{\alpha(c_n-1)} |a(x)| \, dx,
\end{gather*}
where $C_1>0$, $C_2>0$ are some constants. From the relations
\begin{equation*}
    x_{1}^{\alpha(c_1-1)}\ldots x_{n}^{\alpha(c_n-1)} a(x) \in L^1 \left( \mathbb{R}^n_{+} \right)
\end{equation*}
it follows that
\begin{equation*}
    \int\limits_{ \mathbb{R}^n_{+} \cap \mathbf{B}^n(0,C_{2} p_0)} x_1^{\alpha(c_1-1)} \ldots x_n^{\alpha(c_n-1)} |a(x)| \, dx \to 0, \;\;\; p_0 \to +0.
\end{equation*}
It proves that $\frac{\partial \Pi_{\alpha}[a]}{\partial p_0}(p,+0) = 0$. Thus, the equalities~\eqref{inv.piinitcond} are proved and the theorem is proved.
\end{proof}

\section*{References}
    \begin{itemize}
    \item[\lbrack Bo\rbrack] V. I. Bogachev, Measure theory I, Springer, Berlin, 2007
    \item[\lbrack HS\rbrack] G. M. Henkin, A. A. Shananin, Bernstein theorems and Radon transform. Application to the theory of production functions, Trans. Math. Mon. 81, 1990, pp. 189--223
    \item[\lbrack KF\rbrack] A. N. Kolmogorov, S. V. Fomin, Elements of the theory of functions and functional analysis, Dover Publications, 1999
    \item[\lbrack KP\rbrack] S. G. Krantz, H. R. Parks, Geometric integration theory, Birkh\"auser, Boston, 2008   
    \item[\lbrack LL\rbrack] E. H. Lieb, M. Loss, Analysis, Providence, AMS, 2001
    \item[\lbrack Sc\rbrack] L. Schwartz, Th\'eorie des distributions, Hermann, Paris, 1978
    \item[\lbrack Sh\rbrack] A. A. Shananin, The investigation of the generalized model of a pure industry, Matem. Mod. 9 (10), 1997, pp. 73--82 (in Russian)
    \item[\lbrack Yo\rbrack] K. Yoshida, Functional Analysis, Springer, 1980
    \end{itemize}

\end{document}